\documentclass[preprint,12pt]{elsarticle}
\usepackage{amssymb}
\usepackage{amsthm}
\usepackage{amsmath}
\usepackage[T1]{fontenc}
\usepackage{epsfig}
\usepackage{url}
\usepackage{setspace}
\usepackage{color}
\theoremstyle{plain}

\newtheorem{thm}{Theorem}[section]
\newtheorem{cor}[thm]{Corollary}
\newtheorem{lem}[thm]{Lemma}
\newtheorem{prop}[thm]{Proposition}
\newtheorem{ques}[thm]{Question}
\newtheorem{conj}[thm]{Conjecture}

\theoremstyle{definition}
\newtheorem{exam}[thm]{Example}
\newtheorem{df}[thm]{Definition}
\newtheorem{rem}[thm]{Remark}
\def\cal{\mathcal}
\def\op{\operatorname}
\renewcommand{\phi}{\varphi}
\newcommand{\N}{\mathbb{N}}
\newcommand{\Z}{\mathbb{Z}}
\newcommand{\Q}{\mathbb{Q}}
\newcommand{\C}{\mathbb{C}}

\makeatletter
\let\@@pmod\pmod
\DeclareRobustCommand{\pmod}{\@ifstar\@pmods\@@pmod}
\def\@pmods#1{\mkern4mu({\operator@font mod}\mkern 6mu#1)}
\makeatother

\journal{Journal of Number Theory}

\begin{document}

\begin{frontmatter}

\title{A Fibonacci type sequence with Prouhet-Thue-Morse coefficients}

\author[inst1]{Eryk Lipka}

\affiliation[inst1]{organization={Pedagogical University of Kraków, Institute of
Mathematics},
            addressline={Podchor\k{a}\.zych~2}, 
            city={Krak\'ow},
            postcode={30-084},
            country={Poland}}

\author[inst2]{Maciej Ulas}

\affiliation[inst2]{organization={Jagiellonian University, Faculty of Mathematics and Computer Science, Institute of
Mathematics},
            addressline={{\L}ojasiewicza 6}, 
            city={Krak\'ow},
            postcode={30-348},
            country={Poland}}

\begin{abstract}

Let $t_{n}=(-1)^{s_{2}(n)}$, where $s_{2}(n)$ is the sum of binary digits function. The sequence $(t_{n})_{n\in\N}$ is the well-known Prouhet-Thue-Morse sequence. In this note we initiate the study of the sequence $(h_{n})_{n\in\N}$, where $h_{0}=0, h_{1}=1$ and for $n\geq 2$ we define $h_{n}$ recursively as follows: $h_{n}=t_{n}h_{n-1}+h_{n-2}$. We prove several results concerning arithmetic properties of the sequence $(h_{n})_{n\in\N}$. In particular, we prove non-vanishing of $h_{n}$ for $n\geq 5$, automaticity of the sequence $(h_{n}\pmod{m})_{n\in\N}$ for each $m$, and other results.
\end{abstract}


\end{frontmatter}
\section{Introduction}\label{sec1}

Let $\N$ denote the set of non-negative integers, $\N_{+}$ the set of positive integers and for $k\in\N$ let us write $\N_{\geq k}$ for the set $\{n\in\N:\;n\geq k\}$.

Let $f_{0}=0, f_{1}=1$ and for $n\geq 2$ define $f_{n}=f_{n-1}+f_{n-2}$. The sequence $(f_{n})_{n\in\N}$ is the famous Fibonacci sequence with many interesting properties and diverse applications. Many properties of the sequence $(f_{n})_{n\in\N}$ can be deduced from the Binet formula which says that
$$
f_{n}=\frac{1}{\sqrt{5}}(\phi^{n}-(-\phi)^{-n}),
$$
where $\phi=\frac{1}{2}(1+\sqrt{5})$ is the so-called golden ratio. It is the unique positive solution of the equation $x^2-x-1=0$. In particular the Fibonacci sequence has exponential growth. The number of papers devoted to the Fibonacci sequence is enormous. For comprehensive review of its main properties an interested reader may consult the recent book \cite{Ko}. The concept of the Fibonacci sequence was generalized in various directions. One of possible generalizations is to consider the sequence $r_{n}=\varepsilon_{n}r_{n-1}+r_{n-2}$, where $r_{0}=0, r_{1}=1$ and $\varepsilon_{n}\in\{-1,1\}$ is chosen at random with equal probability $1/2$, independently for different values of $n$. Furstenberg \cite{Fu} proved that random recurrent sequences of this kind have exponential growth almost surely. For the random Fibonacci sequence this result was reproved by Viswanath \cite{V} and a simple proof was presented in the work of Makover and McGowan \cite{MM}. 

From the short discussion above, we see that we have two opposite situations: the Fibonacci sequence which has plethora of interesting arithmetic properties and applications, and its random counterpart which is extremely difficult to investigate. However, in-between these two distant objects one can consider a different Fibonacci-like sequence governed by the recurrence sequence $(r_{n})_{n\in\N}$, where $r_{0}=0, r_{1}=1$ and for $n\geq 2$ we put
$$
r_{n}=a_{n}r_{n-1}+r_{n-2}.
$$
Here ${\bf a}=(a_{n})_{n\geq 2}$ is a sequence with values in the set $\{-1, 1\}$. A first generalization which comes to mind is to consider a sequence ${\bf a}$ which is periodic. The behavior of the related sequence $(r_{n})_{n\in\N}$ was investigated in recent papers of Trojovsk\'{y} \cite{Tr} and  Andeli\'{c} et al. \cite{ADF}, where $a_{n}=(-1)^{\lfloor (n-1)/k\rfloor}$ and $k$ is given. In particular, the connection with determinants of certain tri-diagonal matrices was presented and some compact formulas in terms of the Fibonacci sequence were obtained. Moreover, the study of McLellan \cite{Mc} shows that in this case the sequence $(r_{n})_{n\in\N}$ can be quite precisely analyzed (at least its rate of growth rate). In particular, the mentioned results suggest that the sequence $(r_{n})_{n\in\N}$ with periodic sequence ${\bf a}$ is often closely related to the original Fibonacci sequence; it is periodic or has a linear growth. Thus, instead of trying to obtain new results in periodic case we decided to analyze what is going on when the sequence ${\bf a}$ is not periodic and it is not random. The class of such sequences is broad and contains a subclass of {\it automatic sequences}. Let us recall that a sequence ${\bf a}$ is automatic if it can be realized by a finite state machine---so called {\it finite automaton with output}. The class of automatic sequences was introduced and studied by Cobham in 1972. In arithmetic terms, for given $k\in\N_{\geq 2}$ we say that a sequence ${\bf a}$ is $k$-automatic if the set
$$
K_{k}({\bf a})=\{(a_{k^{i}n+j})_{n\in\N}:\;i\in\N \;\mbox{and}\; 0\leq j<k^{i}\},
$$
called the $k$-kernel of ${\bf a}$, is finite. We say that a sequence is automatic if it is $k$-automatic for some $k$. The class of automatic sequences is interesting due to the fact that it can be located between the class of periodic sequences and the class of random (or maybe we should say difficult) sequences. In particular periodic sequences are automatic. Basic (and much more) results concerning automatic sequences can be found in an excellent book of Allouche and Shallit \cite{AS2}.

In our investigations we are interested in one particular interesting 2-automatic sequence, i.e., the Prouhet-Thue-Morse sequence (the PTM sequence for short). The PTM sequence ${\bf t}=(t_{n})_{n\in\N}$ is simply defined as follows: $t_{n}=(-1)^{s_{2}(n)}$, where $s_{2}(n)$ is the sum of binary digits function, i.e., the number of 1's in unique binary expansion of $n$. One can also check that
$$
t_{0}=1,\quad t_{1}=-1,\quad t_{2n}=t_{n},\quad t_{2n+1}=-t_{n}.
$$
Using the above recurrences it is easy to check that the 2-kernel of ${\bf t}$ has the form $K_{2}({\bf t})=\{{\bf t}, -{\bf t}\}$. The PTM sequence has diverse applications in combinatorics, number theory, mathematical analysis and even physics. The ubiquitousness of the PTM sequence is beautifully presented in \cite{AS1}.

Motivated by investigations on random Fibonacci sequences we are interested in the arithmetic properties of the sequence ${\bf h}=(h_{n})_{n\in\N}$, where
$$
h_{0}=0, \quad h_{1}=1, \quad h_{n}=t_{n}h_{n-1}+h_{n-2}.
$$
The first 20 terms of the sequence ${\bf h}$ are the following:
$$
0,1,-1,0,-1,-1,-2,1,-3,-2,-5,3,-2,5,-7,-2,-5,-7,-12,5,\ldots .
$$
We believe that the good understanding of this sequence will give an impetus to study the general class of Fibonacci type sequences with automatic coefficients.

Let us describe the content of the paper in some details. In Section \ref{sec2} we investigate the behavior of the sequence of signs of the sequence ${\bf h}$. In particular, we prove that it is 2-automatic and not periodic. As an application we obtain results concerning the maximal length of the subsequences of consecutive values of ${\bf h}$ which are increasing and decreasing.

In Section \ref{sec3} we consider the sequence ${\bf h}_{m}=(h_{n}\pmod*{m})_{n\in\N}$ for a given $m\in\N_{\geq 2}$. The main result of this section says that ${\bf h}_{m}$ is a 2-automatic sequence for each $m$. Moreover, as an application, we find that for each $m\in\N_{\geq 2}$ the congruence $h_{n}\equiv 0\pmod*{m}$ has infinitely many solutions in positive integers. We also prove that the ordinary generating function of the sequence ${\bf h}$ is transcendental over $\Q(x)$.

Section \ref{sec4} is devoted to the presentation of certain identities involving terms of the sequence ${\bf h}$. We believe that the most interesting part of this section is the result which says that there are infinitely many elements of the sequence ${\bf h}$ which are sums of two squares. This can be seen as a variation on a classical result concerning the identity $f_{2n+1}=f_{n}^2+f_{n+1}^2$ for Fibonacci numbers. Finally, in the last section we offer questions and conjectures which naturally arised during our investigations. We hope that this set of problems will stimulate further research.

\section{Sign behavior of the sequence $(h_{n})_{n\in\N}$ and consequences}\label{sec2}

In this section we analyze sign behavior of the terms of the sequence ${\bf h}$ and present some basic applications.

\begin{lem}\label{lem:tools}
\begin{enumerate}
\item For each $n\in\N_{+}$ we have $h_{2n+1}=t_{2n+1}h_{2n-2}$.
\item For each $n\in\N_{+}$ such that $t_n \neq t_{n-1}$ we have $h_{4n}=t_{n}h_{4n-3}$.
\item The sequence $(u_{n})_{n\in\N}$ of signs of the sequence $(h_{n})_{n\in\N}$, i.e., $u_{n}=\op{sign}(h_{n})$ is $2$-automatic. More precisely,
$$
u_0=u_{3}=0, \quad u_{2n}=-1\; (n \neq 0), \quad u_{2n+1}=t_{n}\; (n\neq 1).
$$
\end{enumerate}
\end{lem}
\begin{proof}
(1) Using known relations $t_{2n} = t_n, t_{2n+1} = -t_n$ we obtain
\begin{align*}
  h_{2n+1} &= t_{2n+1}h_{2n}+h_{2n-1} \\
  &=t_{2n+1}t_{2n}h_{2n-1} + t_{2n+1}h_{2n-2}+h_{2n-1}\\
  &=-t_nt_nh_{2n-1} + h_{2n-1} +t_{2n+1}h_{2n-2}\\
  &=h_{2n-1}(1-t_nt_n) +t_{2n+1}h_{2n-2}\\
  &=t_{2n+1}h_{2n-2}.
\end{align*}
(2) Using a similar reasoning we get
\begin{align*}
  h_{4n} &= t_{4n}h_{4n-1}+h_{4n-2} \\
  &=t_{n}t_{4n-1}h_{4n-2} + t_{n}h_{4n-3}+h_{4n-2}\\
  &=t_nt_{n-1}h_{4n-2} + t_nh_{4n-3} +h_{4n-2}\\
  &=h_{4n-2}(1+t_nt_{n-1}) +t_{n}h_{4n-3}\\
  &=t_{n}h_{4n-3}.
\end{align*}
(3) By a simple calculation, the statement of our lemma holds for $n\le 6$. Assume $h_{2N} < 0$ holds for some $N\in\N_+, N<2n$. Then
$$h_{4n} = t_{4n}h_{4n-1} + h_{4n-2}=t_{4n}h_{4n-1} + t_{4n-2}h_{4n-3} + h_{4n-4}.$$
By (1) $h_{4n}=(1+t_{4n}t_{4n-1})h_{4n-4} + t_{4n-2}t_{4n-3}h_{4n-6}$, after noticing that $t_{4n-3} = -t_{2n-2}=-t_{n-1}=t_{2n-1}=t_{4n-2}$ we obtain
$$h_{4n} < h_{4n-4}(t_{4n}t_{4n-1} + 1)\le 0.$$
Similarly
$$h_{4n+2}=t_{4n+2}h_{4n+1}+h_{4n}=t_{4n+2}t_{4n+1}h_{4n-2}+h_{4n}=h_{4n-2}+h_{4n} < 0.$$
By induction $h_{2n} < 0$ for $n\in \N_+$. For odd $n$ we have
\begin{align*}
  h_{2n+1}&=t_{2n+1}h_{2n}+h_{2n-1}\\
  &=t_{2n+1}t_{2n}h_{2n-1}+t_{2n+1}h_{2n-2}+h_{2n-1}\\
  &= h_{2n-1}(-t_nt_n+1) -t_nh_{2n-2} \\
  &= -t_nh_{2n-2},
\end{align*}
so $u_{2n+1} = t_n$ for $n>1$. Using the fact that $K_2({\bf t}) = \{{\bf t}, -{\bf t}\}$, it is easy to verify that the 2-kernel of $(u_n)_{n\in \N}$ is also finite (in fact, it has seven elements).
\end{proof}

\begin{cor}
Let $u_{n}=\op{sign}(h_{n})$ and write $U(x)=\sum_{n=0}^{\infty}u_{n}x^{n}$. Then
$$
U(x)=-x^2-x^3-\frac{x^{4}}{1-x^2}+x\prod_{n=0}^{\infty}\left(1-x^{2^{n+1}}\right).
$$
In particular, the series $U$ is transcendental over $\Q(x)$.
\end{cor}
\begin{proof}
It is well known that $T(x)=\sum_{n=0}^{\infty}t_{n}x^{n}=\prod_{n=0}^{\infty}(1-x^{2^{n}})$. The series $T$ satisfies a Mahler type functional equation $T(x)=(1-x)T(x^2)$ and it is known that $T$ is transcendental over $\Q(x)$. A simple calculation with the closed form of $u_{n}$ presented in Lemma \ref{lem:tools}(3) reveals the form of $U$. Transcendentality of the series $U$ is a simple consequence of transcendentality of $T$. Indeed, if $U$ were a rational function then
$$
\frac{1}{x}\left(U(x)+x^2+x^3+\frac{x^{4}}{1-x^2}\right)=T(x^2)=\frac{T(x)}{1-x}
$$
would be a rational function too. Thus, $T$ would be rational a function---a~contradiction.
\end{proof}

As a next consequence of our findings we present the following:

\begin{cor}
\begin{enumerate}
\item The set
$$
\{n\in\N:\;h_{n+i}<0\;\mbox{for}\;i=0, 1, 2, 3, 4\}
$$
is infinite.
\item There is no $n\in\N$ satisfying $h_{n+i}<0$ for $i=0, \ldots, 5$.
\item There is no $n\in\N$ satisfying $h_{n+i}>0$ for $i=0, 1, 2$.
\end{enumerate}
\end{cor}
\begin{proof}
The first statement is a simple consequence of the existence of infinitely many $n\in\N$ such that $u_{2n+1}=u_{2n+3}=-1$, i.e., $t_{n}=t_{n+1}=-1$ and the fact that $u_{2n}=u_{2n+2}=u_{2n+4}=-1$.

Because there are no three consecutive equal values in the PTM sequence, there is no $n$ such that $h_{n+i}<0$ for $i=0, \ldots, 5$.

The last statement is clear from the sign behavior of the sequence $(h_{n})_{n\in\N}$.
\end{proof}

\begin{prop}
\begin{enumerate}
\item The set
$$
\{n\in\N:\;h_{n}<h_{n+1}<h_{n+2}\}
$$
is infinite.

\item There is no $n\in\N$ satisfying $h_{n}<h_{n+1}<h_{n+2}<h_{n+3}$.

\item The set
$$
\{n\in\N:\;h_{n}>h_{n+1}>h_{n+2}>h_{n+3}\}
$$
is infinite.

\item There is no $n\in\N$ satisfying $h_{n}>h_{n+1}>h_{n+2}>h_{n+3}>h_{n+4}$.
\end{enumerate}
\end{prop}
\begin{proof}
(1) Take any $n\in \N, n>2$ such that $t_n= t_{n-1}=1$ (there are infinitely many such values). From Lemma \ref{lem:tools}(3) we have
\begin{align*}
 u_{4n-1} &= t_{2n-1} = -t_{n-1} = -1,\\
 u_{4n+1} &= t_{2n} = t_n = 1.
\end{align*}
From the recurrence relations satisfied by the sequence ${\bf h}$ we obtain
\begin{align*}
  h_{4n+3} &= t_{4n+3}h_{4n+2} + h_{4n+1}= h_{4n+2}+h_{4n+1}> h_{4n+2},\\
  h_{4n+4} &= t_{4n+4}h_{4n+3} + h_{4n+2}= -h_{4n+3}+h_{4n+2}.
\end{align*}
From Lemma \ref{lem:tools}(1) we deduce
$$h_{4n+4} = -h_{4n+1} = -t_{4n+1}h_{4n}-h_{4n-1} =h_{4n}-h_{4n-1} > h_{4n}.$$
But also $h_{4n} = t_{4n+3}h_{4n}=h_{4n+3}$ and we get $h_{4n+2} < h_{4n+3} < h_{4n+4}$.

(2) The statement holds for $n\leq 5$. First, assume there is a $n\in \N_+$ with
$$h_{2n} < h_{2n+1}<h_{2n+2}<h_{2n+3}.$$
From Lemma \ref{lem:tools}(1) we get $h_{2n+3}=t_{2n+3}h_{2n}=-t_{n+1}h_{2n}$; we know that $h_{2n} < h_{2n+3}$, so $t_{n+1} = 1$.
From the recurrence
$$h_{2n+2}=t_{2n+2}h_{2n+1} + h_{2n} = t_{n+1}h_{2n+1} + h_{2n}= h_{2n+1} + h_{2n},$$
and as $h_{2n}<0$, we get the inequality $h_{2n+2} <h_{2n+1}$, which contradicts our assumption.

Secondly, assume there is $n\in \N_+$ with
$$h_{2n+1}<h_{2n+2}<h_{2n+3} < h_{2n+4}.$$
From Lemma \ref{lem:tools}(3) 
$t_{n+1} = u_{2n+3} \le u_{2n+4}=-1$. And again from the recurrence relation 
$$h_{2n+3}=t_{2n+3}h_{2n+2} + h_{2n+1} = -t_{n+1}h_{2n+2} + h_{2n+1} =h_{2n+2} + h_{2n+1},$$
and as $h_{2n+1}<0$ we get a contradictory inequality $h_{2n+3} <h_{2n+2}$.

It is necessary to also check the quadruplets starting from $n\in \{0,1\}$.

(3) Take any $n\in \N, n>2$ such that $t_n= t_{n-2}=-1, t_{n-1}=1$ (there are infinitely many values of $n$ with this property). From Lemma \ref{lem:tools} we have the following chain of equalities:
\begin{align*}
    h_{4n-4} &= t_{4n-4}h_{4n-7} = h_{4n-7},\\
    h_{4n-3} &= t_{4n-3}h_{4n-6} = -h_{4n-6},\\
    h_{4n-1} &= t_{4n-1}h_{4n-4} = h_{4n-4} = h_{4n-7},\\
    h_{4n} &= t_{4n}h_{4n-3} = -h_{4n-3} = h_{4n-6},\\
    u_{4n-5} &= t_{2n-3}=1,\\
    u_{4n-1} &= t_{2n-1}=-1,\\
    u_{4n} &= -1.
\end{align*}
Moreover, we also get
\begin{align*}
    h_{4n-5} &= t_{4n-5}h_{4n-6} + h_{4n-7} = -h_{4n-6} + h_{4n-7},\\
    h_{4n+1} &= t_{4n+1}h_{4n} + h_{4n-1} = h_{4n} + h_{4n-1},\\
    h_{4n+2} &= t_{4n+2}h_{4n+1} + h_{4n} = h_{4n+1} + h_{4n}.
\end{align*}
Hence
\begin{align*}
    h_{4n-1} - h_{4n} &= h_{4n-5} > 0,\\
    h_{4n} - h_{4n+1} &= -h_{4n-1} > 0,\\
    h_{4n+1} - h_{4n+2} &= -h_{4n} > 0,
\end{align*}
which implies $h_{4n-1}>h_{4n}>h_{4n+1}>h_{4n+2}$.

(4) Let $n \in \N_{\geq 4}$ be chosen such that $h_n <0$. By recurrence if $t_{n+2}=-1$ then $h_{n+2} = h_n - h_{n+1}>0 > h_n$. By a similar argument we get $t_{n+2}=t_{n+3}=t_{n+4}=1$ which is impossible as the PTM sequence is cube-free.

Otherwise, if $h_n>0, n\ge3$ we deduce from Lemma \ref{lem:tools}(3) that $n$ is odd and $h_{n+1}<0$. By Lemma \ref{lem:tools}(1) $h_{n+4} = t_{n+4}h_{n+1} \ge h_{n+1}$, a contradiction.
It remains to check $n \in \{0,1,2\}$.
\end{proof}

As a different application we prove that in size the negative terms in the sequence ${\bf h}$ dominate positive values. More precisely, we have the following

\begin{lem}
For $n>3$ we have $\sum_{i=0}^{n}h_{i}<0$.
\end{lem}
\begin{proof}
As ${\bf h} = (0,1,-1,0,-1,-1,-2,1,\ldots)$ the statement is true for $n<8$. For any $n\ge 8$ with $h_n >0$ (the first one is $h_{11}=3$) it follows from Lemma \ref{lem:tools} that $n$ is odd and $h_n = -h_{n-3}$. Hence when $n\ge 8$ we can write
\begin{align*}
  \sum_{i=0}^{n}h_{i}&=\sum_{i=0}^{7}h_{i} + \sum_{\substack{i=8 \\ h_i>0}}^{n}h_{i} + \sum_{\substack{i=8 \\ h_{i+3}>0}}^{n}h_{i} + \sum_{\substack{i=8 \\ h_i<0 \\ h_{i+3}<0}}^{n}h_{i}\\
  &=-3 + \sum_{\substack{i=11 \\ h_{i}>0}}^{n}(h_{i} + h_{i-3})+\sum_{\substack{i=n-2 \\ h_{i+3}>0}}^n h_{i}+\sum_{\substack{i=8 \\ h_i<0 \\ h_{i+3}<0}}^{n}h_{i}\\
  &= -3 + 0 +\sum_{\text{some }i \text{ with } h_i<0} h_i<0
\end{align*}
and hence the result.
\end{proof}

\section{Automaticity of the sequence $(h_{n}\pmod*{m})_{n\in\N}$}\label{sec3}

A well-known property of the Fibonacci sequence $(f_{n})_{n\in\N}$ is that the congruence $f_{n}\equiv 0\pmod*{m}$ has infinitely many solutions for each $m\in\N_{\geq 2}$. In this section we analyze the behavior of the sequence ${\bf h}_{m}=(h_{n}\pmod*{m})_{n\in\N}$. In particular, we prove 2-automaticity of ${\bf h}_{m}$ and apply this property to the proof that the ordinary generating function of the sequence ${\bf h}$ is transcendental over $\Q(x)$.

Because $t_{n}\equiv 1\pmod*{2}$ we immediately get $h_{n}\equiv F_{n}\pmod*{2}$. As a consequence we get the following
\begin{lem}
For $n\in\N$ we have $h_{n}\equiv 0\pmod*{2}\;\Longleftrightarrow\; n\equiv 0\pmod*{3}$.
\end{lem}
\begin{proof}
This is a consequence of the well-known modulo 2 behavior of the Fibonacci sequence.
\end{proof}

To obtain the main result of this section we need to recall one (of the many) equivalent definitions of automaticity of infinite sequences. To do this we recall the definition of an automaton.  

\begin{df} A 6-tuple $\mathcal A = (Q,\Sigma,\delta,q_0,\Delta,\tau)$ is a {\it deteriministic finite automaton with output} (or DFAO) if it has the following properties:
\begin{itemize}
    \item the set of states $Q$ is finite;
    \item the input alphabet $\Sigma$ is a finite set;
    \item the transition function $\delta$ is a mapping $Q\times \Sigma \mapsto Q$;
    \item the starting state $q_0$ is a element of $Q$;
    \item the output alphabet $\Delta$ is a finite set;
    \item the output function $\tau$ is a mapping $Q\mapsto \Delta$.
\end{itemize}
The transition function is usually extended to $\hat \delta : Q\times \Sigma^* \mapsto Q$ by defining $\hat \delta(q,\varepsilon) = q, \hat\delta(q,sw) = \hat\delta(\delta(q,s),w)$ for any $q\in Q, s\in \Sigma, w\in \Sigma^*$. Obviously $\hat \delta(q,s) = \delta (q,s)$ for $(q, s)\in Q\times\Sigma$.

Such an automaton represents a {\it finite-state function} $f_{\mathcal A}: \Sigma^* \mapsto \Delta$ by defining for any finite word $w=w_1w_2w_3\ldots w_i$,
$$f_{\mathcal A}(w) = \tau(\delta(\ldots \delta(\delta(\delta(q_0,w_1),w_2),w_3),\ldots, w_i)) = \tau(\hat\delta(q_0,w)).$$
\end{df}

\begin{df}
A sequence $(a_n)_{n\in\N}\in\Delta^{\N}$ is $k$-automatic if $a_n$ is a finite-state function of the base-$k$ digits of $n$.
\end{df}

The digits of $n$ are being fed to the automaton from left to right, i.e. starting with the most significant digit. It is worth noting that reversing the direction would yield an equivalent definition: for any DFAO $\mathcal A$ there is another DFAO $\mathcal B$ such that $f_{\mathcal A}(w) = f_{\mathcal B}(w^R)$ for any $w\in \Sigma^*$ \cite[Theorem 5.2.3]{AS2}.

DFAOs are commonly represented by their {\it transition graph}. Each state is represented by a node, $\tau$ is represented by labels on states, and for any $q\in Q$ and $s\in \Sigma$ there is an edge from $q$ to $\delta(q,s)$ with label $s$. An example is given in Figure \ref{automaton_tm}.

\begin{figure}[h]
       \centering
         \includegraphics[height=1in]{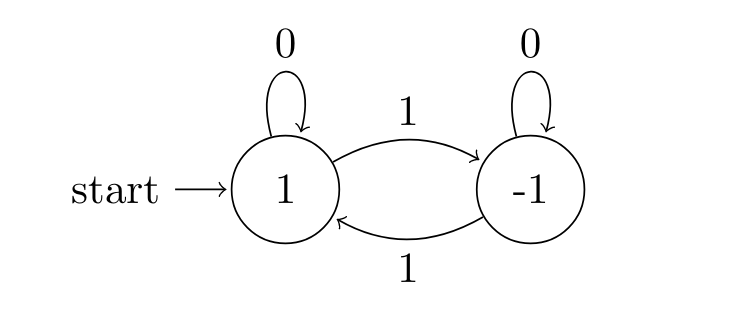}
        \caption{Transition graph of DFAO representing $(t_n)_{n \in \N}$}
        \label{automaton_tm}
\end{figure}

\begin{thm}
Let $m\in \N_{\ge 2}$. The sequence ${\bf h}_{m}=(v_{m,n})_{n\in N}$, where $v_{m,n}=h_{n} \pmod*{m}$, is 2-automatic.
\end{thm}
\begin{proof}
Because we fixed $m$ and will not change it during the proof, in order to simplify the notation we write $v_{n}$ instead of $v_{m,n}$.

For $k\in\N$ we define a sequence of functions
\begin{align*}
    f_k : \left(\{0,\ldots, m-1\}\times \{0,\ldots, m-1\} \right)&\times \{-1,1\}  \\ \mapsto\{0,\ldots, m-1\}&\times \{0,\ldots, m-1\},
\end{align*}
recursively as follows:
\begin{align*}
f_0((x,y),z) &= (y,zy+x \pmod*{m}),\\
f_{i+1}((x,y),z) &= f_i(f_i((x,y),z),-z\pmod*{m})\text{ for } i\in \N.
\end{align*}
From the recurrence relations of $t_n$ it is easy to see that for any $l\in \N$,
\begin{align*}
    f_0((v_{l-2}, v_{l-1}), t_l) &= (v_{l-1}, v_{l}\pmod*{m}),\\
    f_1((v_{2l-2}, v_{2l-1}), t_l) &= (v_{2l}, v_{2l+1}\pmod*{m}),\\
    f_2((v_{4l-2}, v_{4l-1}), t_l) &= (v_{4l+2}, v_{4l+3}\pmod*{m}).
\end{align*}

By a simple induction we find that for any $k\in \N$,
$$f_k(v_{2^kl-2}, v_{2^kl-1}, t_l) = (v_{2^kl +2^k-2}, v_{2^kl+2^k-1}\pmod*{m}).$$
There are only finitely many possible functions $f_k$, so we can take the smallest $p\in \N$ and the smallest $q \in \N_+$ such that $f_p \equiv f_{p+q}$ (as functions). The function $f_{p+q}$ can be written in terms of $f_p$ composed $2^q$ times. The same expression can be used to define $f_{p+2q}$ in terms of $f_{p+q}$ and so on. It follows that $f_p\equiv f_{p+q} \equiv f_{p+2q}\equiv \ldots$, and $f_{p+k}\equiv f_{p+q+k} \equiv f_{p+2q+k}\equiv \ldots$ for $k \in \N$. As a consequence, the sequence $(f_k)_{k\in \N}$ is ultimately periodic.

For any $n\in \N_+$ there exist $a_1,\cdots, a_b \in \N$ with $a_1 >a_2 > \cdots > a_b$ such that $n=2^{a_1} + 2^{a_2} + \ldots + 2^{a_b}$ (the base 2 representation). We can write
\begin{align*}
  f_{a_1}((v_{-2},v_{-1}),t_0)&=(v_{2^{a_1}-2},v_{2^{a_1}-1}\pmod*{m}), \\
  f_{a_2}((v_{2^{a_1}-2},v_{2^{a_1}-1}),t_{2^{a_1-a_2}})&=(v_{2^{a_1}+2^{a_2}-2},v_{2^{a_1}+2^{a_2}-1}\pmod*{m}), \\
  f_{a_3}((v_{2^{a_1}+2^{a_2}-2},v_{2^{a_1}+2^{a_2}-1}),t_{2^{a_1-a_3}+2^{a_2-a_3}})&=(v_{2^{a_1}+2^{a_2}+2^{a_3}-2},v_{2^{a_1}+2^{a_2}+2^{a_3}-1}\pmod*{m}),
\end{align*}
  and so on. As $t_n = (-1)^{s_2(n)}$ we can compute
  \begin{align*}
      f_{a_b}(f_{a_{b-1}}(\ldots f_{a_3}(f_{a_2}(f_{a_1}(v_{-2},v_{-1},1),-1),1)& \ldots,(-1)^b),(-1)^{b-1})\\&=(v_{n-2},v_{n-1}\pmod*{m}).
  \end{align*}

We define $a'_1,\ldots, a'_b \in \N$ by
$$
a'_k =\begin{cases} a_k & \text{ if } k<p, \\
((a_k-p) \pmod*{q}) +p &\text{ otherwise} ,\end{cases}
$$
so that $f_{a'_k}\equiv f_{a_k}$ for $k\in \{1,\ldots,b\}$.

We are ready to construct a DFAO $\mathcal A = (Q,\Sigma,\delta,q_0,\Delta,\tau)$, which given the binary representation of $n$ (reading from the most significant digit) will compute $v_n$. Obviously $\Sigma = \Delta = \{0,1,2,\ldots, m-1\}$.

\begin{figure}[h]
       \centering
         \includegraphics[height=3in]{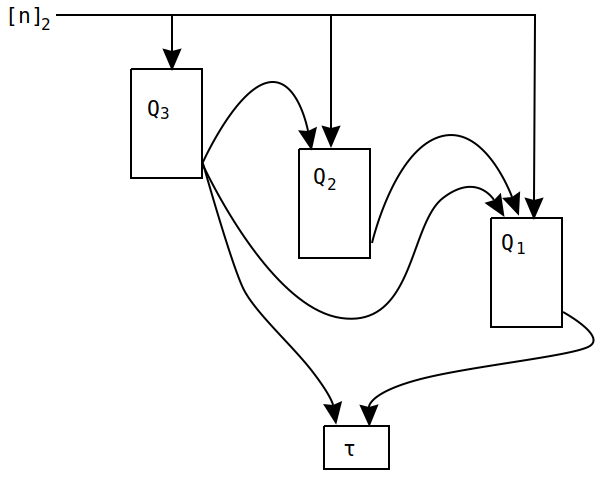}
        \caption{Idea of a cascade automaton}
        \label{cascade}
\end{figure}

The type of construction we will perform is a variation of the so-called cascade product of automatons. Let us define
\begin{equation*}
    Q_1 = (\Sigma\times\Sigma)^p,\quad Q_2 = (\Sigma\times\Sigma)^q,\quad Q_3 = \{-1,1\},\quad Q = Q_1 \times Q_2 \times Q_3.
\end{equation*}
We see that the set of states has some internal structure. We consider $Q_1,Q_2$ to be composed``cell'', each cell can be in one of $m^2$ states representing a pair of elements of $\Sigma$. Every element of $Q_1$ represents a state of $p$ such cells, and elements of $Q_2$ represent a state of $q$ such cells. Looking at  Figure \ref{cascade} we see that our automaton can be split into parts. The leftmost part is the automaton for the PTM sequence, which after being fed the binary representation of $n$ is in state $t_n$. The middle receives as input not only binary digits of $n$ but also previous state of the left automaton, and the rightmost automaton receives digits of $n$ and previous state of remaining automatons. The initial state of our automaton is $q_0 =((v_{-2},v_{-1})^p,(v_{-2},v_{-1})^q,1)$.

We have to define the transition function. On each step values stored in cells of $Q_2$ will be ``circularly shifted around'' and values stored in $Q_1$ will be ``moved downwards''.
If the read digit is zero, the transformation only consists of moving values between the cells of $Q_1, Q_2$ as follows:
\begin{align*}
\delta&\big(((g_0,\ldots, g_{p-1}),(g_p,\ldots,g_{p+q-2}, g_{p+q-1}),c),0\big)=\\
&((g_1,\ldots, g_p),(g_{p+1},\ldots,g_{p+q-1},g_{p}),c).
\end{align*}
However, when we encounter a 1, we apply the following proper transformation before moving:
\begin{align*}
  \delta&\big(((g_0,\ldots, g_{p-1}),
  (g_p,\ldots,g_{p+q-2},g_{p+q-1}),c),1\big)=\\
  &\big((f_0(g_1,c),\ldots, f_{p-1}(g_p,c)),
  (f_{p}(g_{p+1},c),\ldots,f_{p+q-2}(g_{p+q-1},c),f_{p+q-1}(g_{p},c)), -c\big).
\end{align*}

\begin{figure}[h]
       \centering
         \includegraphics[height=3in]{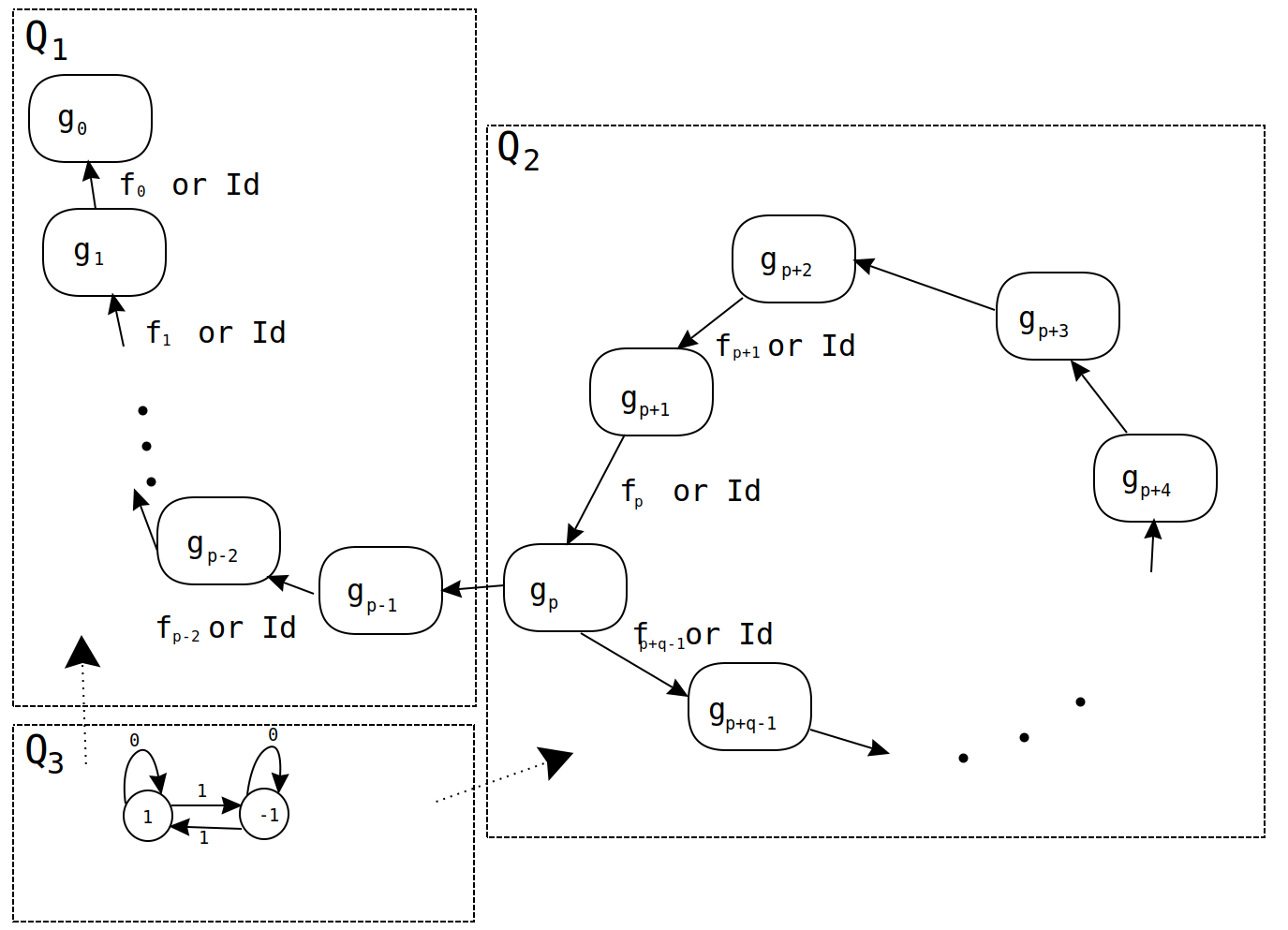}
        \caption{On each step the values are moved between cells; if the digit is $1$ the proper functions are applied when moving. $Q_3$ is a simple two-state Thue-Morse automaton; its current value is passed as the last parameter to $f_0,\ldots, f_{p+q-1}$}
\end{figure}

When reading the binary representation of $n$ we first encounter a digit ``1'' on the position $a_1$; after processing that step there will be the pair $(v_{2^{a_1}-2},v_{2^{a_1}-1}) = f_{a_1}((v_{-2}, v_{-1}), t_0)$ stored in the $a'_1$th cell. The following zeros will only move the values until the moment we encounter a digit ``1'' on the position $a_2$. Before reading that ``1'', the pair $(v_{2^{a_1}-2},v_{2^{a_1}-1})$ will be in the $(a_2+1)$th cell (or $p$th when $a_2=p+q-1$).
After reading that ``1'', the pair $$(v_{2^{a_1}+2^{a_2}-2},v_{2^{a_1}+2^{a_2}-1})=f_{a_2}((v_{2^{a_1}-2},v_{2^{a_1}-1}), t_{2^{a_1-a_2}})$$
will be stored in the $a'_2$th cell.
The process will continue until we run out of digits; then $(v_{n-2}, v_{n-1})$ is in the $0$th cell.  It follows that

$$  \hat \delta\big(q_0,[n]_2\big)=((v_{n-2},v_{n-1}),\ldots,t_n).  $$

It remains to define the output function such that the automaton returns $v_n\equiv t_nv_{n-1}+v_{n-2}\pmod*{m}$, i.e, we put
$$\tau\big((g_{0,0}, g_{0,1}),(g_{1,0}, g_{1,1}),\ldots ,(g_{q+p-1,0}, g_{q+p-1,1}),c\big) = cg_{0,1} + g_{0,0}.$$
\end{proof}

\begin{exam}
Let us follow the above construction in the case $m=3$. Let us recall that $v_{n}=v_{3,n}$ in this case. We have
\begin{align*}
    f_0(x,y,z)&\equiv (y, zy+x),\\
    f_1(x,y,z)& \equiv f_0(y,zy+x,-z)\equiv (zy+x, -z^2y-zx+y)\\&\equiv (zy+x, -zx),\\
    f_2(x,y,z) &\equiv f_1(zy+x, -zx,-z)\equiv (z^2x+zy+x, z^2y+zx)\\&\equiv (zy-x, zx+y),\\
    f_3(x,y,z) &\equiv f_2(zy-x, zx+y,-z)\equiv (-z^2x-zy-zy+x, -z^2y+zx+zx+y)\\&\equiv (zy, -zx),\\
    f_4(x,y,z) &\equiv f_3(zy,-zx,-z)\equiv (z^2x, z^2y)\\&\equiv (x, y),\\
    f_5(x,y,z) &\equiv f_4(x,y,-z) \\&\equiv (x,y),\\
    \vdots &
\end{align*}
In this case we have $p=4,q=1$, so the internal state of an automaton can be described by a tuple $$(g_0,g_1,g_2,g_3,g_4,c) \in (\{0,1,2\}\times\{0,1,2\})^5 \times \{-1,1\}.$$
Because $f_4(x,y,z)\equiv (x,y)$ the value of $g_4$ will be constant and equal to the initial value of $(2,1)$. It follows that the output value depends only on the last four digits of $n$ and the value of $c$ just before reading those four digits, or in other words on $n \mod 16$ and $s_2(\lfloor \frac{n}{16} \rfloor) \mod 2$.

After a direct computation one can see that the sequence $(v_{n})_{n\in \N}$ is the image of $(t_{n})_{n\in \N}$ under the morphism
\begin{align*}
  1 &\rightarrow 0 1 2 0 2 2 1 1 0 1 1 0 1 2 2 1  ,\\
  -1 &\rightarrow 1 2 0 2 2 0 2 2 1 1 0 1 2 0 2 1.
\end{align*}

We see that in general, the number of states of this automaton is $9^5\cdot 2=118098$, it may be of course reduced using some minimalization algorithms. By observing the behavior of $v_{n}$ for the first $10^5$ values one can obtain an automaton with 28 states, which we present in Figure \ref{automaton_3} below.

\begin{figure}[h]
       \centering
         \includegraphics[width=4.5in]{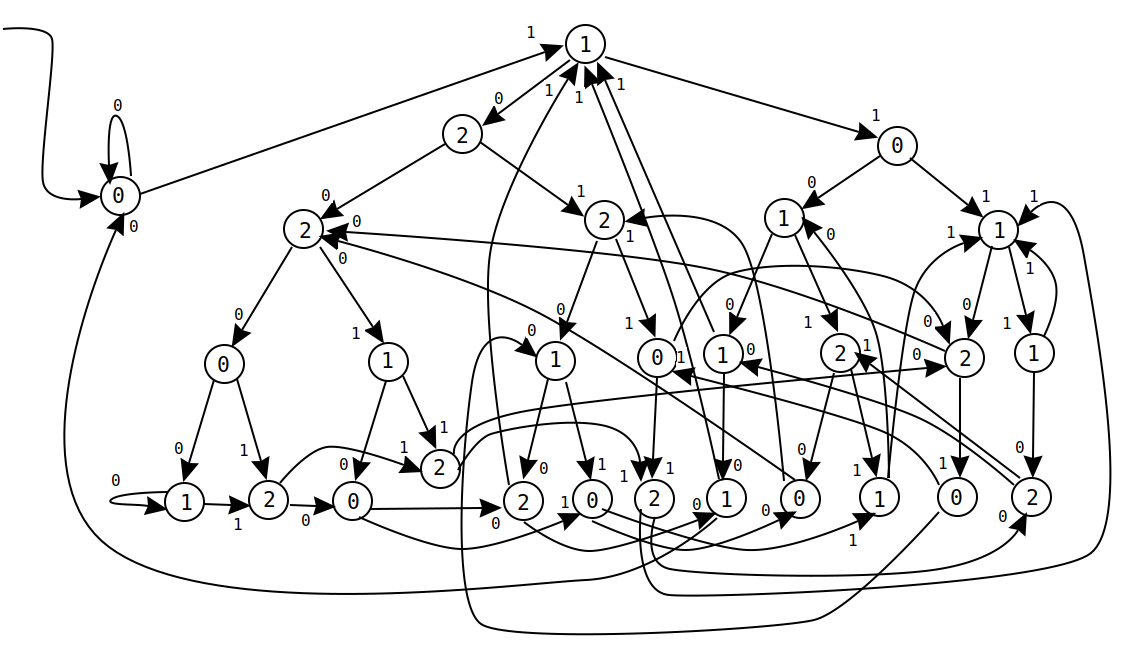}
        \caption{An automaton generating the sequence $(v_{n})_{n\in\N}$}
       \label{automaton_3}
    \end{figure}

However, the ``cyclic'' part of the construction is small ($q=0$), and it is easier to construct an automaton reading from the least significant digit. The automaton will need to differentiate between the possible values of $n \mod 16$, and after that compute parity of remaining digits. It can be done using 47 states: a binary tree of 31 states and an additional state at every leaf so we get a binary counter. This can be further minimized to obtain a 22-state automaton in Figure \ref{automaton_3_reversed}.

\begin{figure}[h]
       \centering
         \includegraphics[width=4.9in]{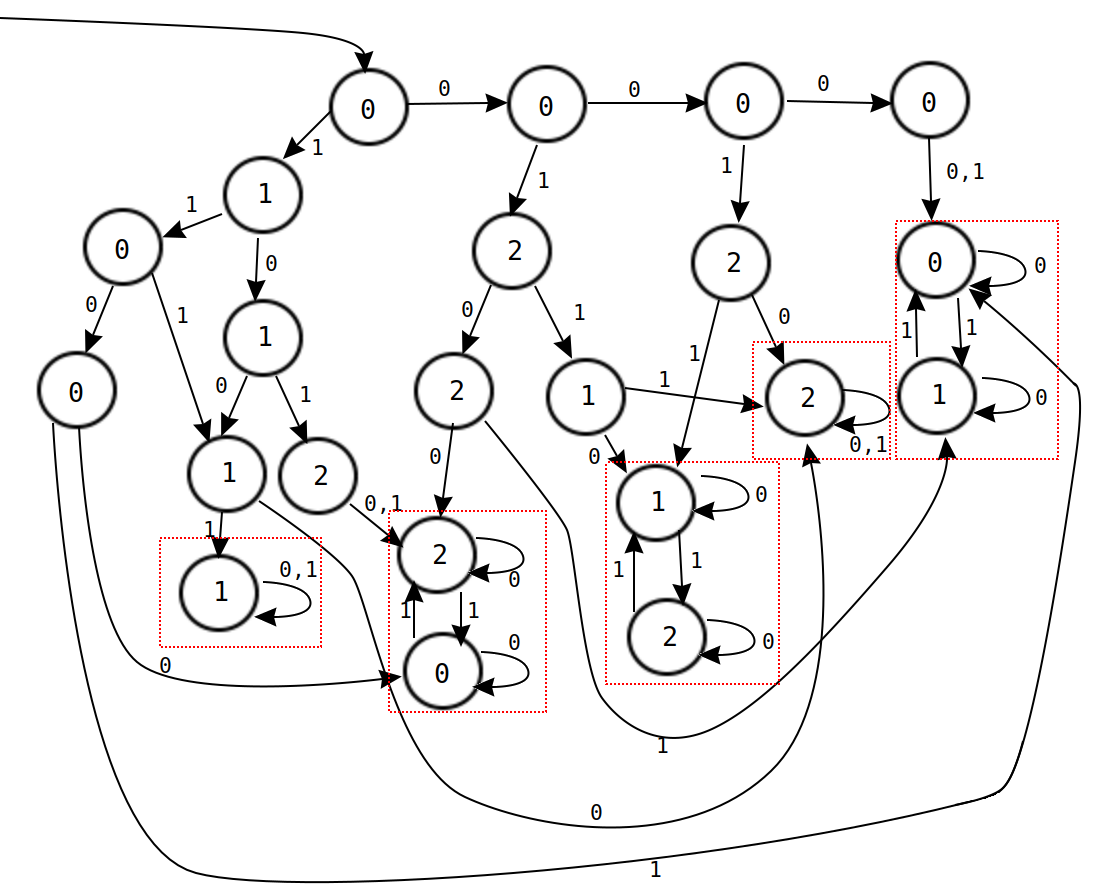}
        \caption{A minimal automaton generating the sequence $(v_{n})_{n\in\N}$ reading from the least significant digit. The dotted lines represent five different ``leaves'' of the original tree, after feeding the automaton with the first four digits we always finish in one of those parts and only parity of remaining 1's is relevant to the output}
       \label{automaton_3_reversed}
    \end{figure}

\end{exam}

\begin{exam}
Let us follow the above construction in the case $m=4$. Let us recall that $v_{n}=v_{4,n}$. We have
\begin{align*}
    f_0(x,y,z)&\equiv (y, zy+x),\\
    f_1(x,y,z)& \equiv f_0(y,zy+x,-z)\equiv (zy+x, -z^2y-zx+y)\\&\equiv (zy+x, -zx),\\
    f_2(x,y,z) &\equiv f_1(zy+x, -zx,-z)\equiv (z^2x+zy+x, z^2y+zx)\\&\equiv (2x+zy, zx+y),\\
    f_3(x,y,z) &\equiv f_2(2x+zy, zx+y,-z)\equiv (4x+2zy-z^2x-zy, -z^2y-2zx+zx+y)\\&\equiv (-x+zy, -zx),\\
    f_4(x,y,z) &\equiv f_3(-x+zy,-zx,-z)\equiv (x-zy+z^2x, -zx+z^2y)\\&\equiv (2x-zy, -zx+y),
    \end{align*}
    \begin{align*}
    f_5(x,y,z) &\equiv f_4(2x-zy,-zx+y,-z) \equiv (4x-2zy-z^2x+zy,2zx-z^2y-zx+y) \\&\equiv (-x-zy,zx),\\
    f_6(x,y,z) &\equiv f_5(-x-zy,zx,-z) \equiv (x+zy+z^2x,zx+z^2y) \\&\equiv (2x+zy,zx+y) = f_2(x,y,z),\\
    \vdots &
\end{align*}
In this case we have $p=2,q=4$, so the internal state of an automaton can be described by a tuple $$(g_0,g_1,g_2,g_3,g_4,g_5,c) \in (\{0,1,2,3\}\times\{0,1,2,3\})^6 \times \{-1,1\}.$$
Again we can split the behavior depending on $n \pmod*{4}$ and $\lfloor \tfrac n 4\rfloor$. After processing all but $p=2$ digits of $n$ we have some values in $(g_2,c)$, and $v_{n}$~can be computed in terms of this value and the last two binary digits of~$n$. Take the sequence $(v'_{n})_{n\in \N}$ such that $v'_n$ is equal to the pair $(g_2,c)$ after our automaton was fed with $n$. Hence $(u_n)_{n\in \N}$ is an image of $(v'_{n})_{n\in N}$ under a uniform 4-morphism. On the other hand, the state of $(g_2,c)$ depends on the state in which $(g_2,c)$ was $q=4$ digits prior and the value of those digits; this leads to a simple left-to-right 16-automaton with at most $4\cdot 4 \cdot 2$ states that computes $(u'_{n})_{n\in N}$.
Again, using experimental approach one can find dependencies between appropriate subsequences $(v_{2^{k}n+i})_{n\in\N}$ to construct an automaton which represents $(v_{n})_{n\in\N}$. Sadly, the minimal automaton has 80~states, and there is no readable way to represent it on a picture.

In a DFA we have to stick to the one way of inputing digits; this is why our construction quickly gets big as we need to perform computations for all possible cases until we can verify which case we are in. In the left-to-right approach we have to guess how many digits $n$ has $\pmod*{q}$ and when the number of digits left is smaller than $p$.  In the right-to-left approach, we would need to guess $t_n$ and use inverses of recurrence relations, which may be multi-functions.
\end{exam}

We give two applications of our results.

\begin{thm}\label{mvalues}
 For each $m\in\N_{\geq 2}$ there are infinitely many solutions of the congruence $h_{n}\equiv 0\pmod*{m}$.
\end{thm}
\begin{proof}
Repeat the construction from the previous proof and notice that if $f_p\equiv f_{p+q}$ then
$$f_{p+q}(v_{-2}, v_{-1},1)=(v_{2^{p+q}-2},v_{2^{p+q}-1})=f_p(v_{2^{p+q}-2^p-2},v_{2^{p+q}-2^p-1},t_{2^q-1}).$$
Without loss of generality, we can take $q$ to be even (take $2q$ instead of $q$) so $t_{2^q-1} =1$. The function $f_p(\cdot, \cdot, 1)$ is bijective as a composition of bijective functions, hence $$(v_{2^{p+q}-2^p-2},v_{2^{p+q}-2^p-1})=(v_{-2},v_{-1})$$
and $0=v_0=v_{2^{p+q}-2^p}$. There are infinitely many $p,q$ such that $f_p\equiv f_{p+q}$ and $q$ is even, hence there are infinitely many 0's in this sequence.
\end{proof}
The above proof also shows that for any $m$ all values taken by the sequence $(h_n \mod m)_{n\in \N}$ appear infinitely many times.

The second application is devoted to the proof of transcendentality of the ordinary generating function of the sequence ${\bf h}$. Let
$$
H(x)=\sum_{n=0}^{\infty}h_{n}x^{n}.
$$

We prove the following

\begin{thm}
The series $H$ is transcendental over $\Q(x)$.
\end{thm}
\begin{proof}
The idea of the proof is simple. Let us suppose that $H$ is not transcendental over $\Q(x)$. This means that $H$ is algebraic over $\Q(x)$ and the same is true for the reduction modulo 3 of $H$. In other words, the power series
$$
G(x)=H(x)\pmod*{3}=\sum_{n=0}^{\infty}h_{n}\pmod*{3}x^{n}
$$
is algebraic over $\Q(x)$. However, the sequence $(v_{n})_{n\in\N}$, where $v_{n}=h_{n}\pmod*{3}$, is 2-automatic and one can check that the series $G$ satisfies the Mahler type functional equation $p(x)G(x^2)+q(x)G(x)+r(x)=0$, where
\begin{align*}
p(x)=&\;(1+x)(x^2-1)(1+x^4)(1+x^8)(1-x^{32})\times,\\
     &\;(2 x^{12}-3 x^{11}+2 x^{10}-x^9+x^8-2 x^7+x^6-2 x^5+4 x^4-4 x^3+2 x^2-1),\\
q(x)=&\;(x^{32}-1)\times \\
     &(2 x^{24}-3 x^{22}+2 x^{20}-x^{18}+x^{16}-2 x^{14}+x^{12}-2 x^{10}+4 x^8-4 x^6+2 x^4-1),\\
r(x)=&\;2 x^{57}-x^{56}+3 x^{55}+4 x^{54}-2 x^{53}-3 x^{52}+2 x^{51}-3 x^{50}-x^{49}+2 x^{48}\\
     &\;+5x^{47}+x^{46}-9 x^{45}+6 x^{43}-5 x^{42}-9 x^{41}+x^{40}+5 x^{39}+4 x^{38}-5 x^{37}\\
     &\;-3 x^{36}-5x^{34}+4 x^{32}-x^{30}-4 x^{29}-2 x^{28}+x^{27}-5 x^{26}-2 x^{25}+x^{24}+\\
     &\;-5x^{23}+x^{22}-x^{21}-x^{20}-4 x^{19}-3 x^{18}-2 x^{17}-x^{14}-x^{13}+2 x^{12}-4 x^{10}\\
     &\;+7 x^9+2x^8-6 x^7+3 x^5-3 x^4+x^3-x^2-x.
\end{align*}
Note that we treat the series $G$ as an element of $\Z[[x]]\subset \C[[x]]$. Let us recall the Nishioka theorem which says that if $f\in \C[[x]]$ satisfies the functional equation $f(x^{d})=\phi(x,f(x))$, where $\phi\in \C(x,y)$, then $f$ is rational or transcendental \cite{Ni1} (see also \cite{Ni2}). In our case $d=2$ and
$$
\phi(x,y)=-\frac{r(x)+q(x)y}{p(x)}.
$$
Thus, to get the result it is enough to prove that the series $G$ is not rational. To see this it is enough to consider the reduction modulo 2 of $G$, i.e., the series $F(x)=G(x)\pmod*{2}$. The series $F$ is an element of $\Z_{2}[[x]]$ and satisfies the algebraic equation $p(x)G(x)^2+q(x)G(x)+r(x)=0$ (this is a consequence of the congruence $G(x^2)\equiv G(x)^2\pmod*{2}$). However, it is easy to check that the polynomial $P(x,y)=p(x)y^2+q(x)+r(x)$ is irreducible as an element of $\Z_{2}[x][y]$ and hence the series $F$ cannot be rational over $\Z_{2}[x]$. In consequence the series $G$ is not rational and the Nishioka theorem implies its transcendentality over $\C(x)$. Thus, the same is true for the series $H$ and our theorem is proved.
\end{proof}

\section{Some identities}\label{sec4}

The Fibonacci sequence satisfies many interesting and sometimes unexpected identities. Most of them can be proved with the help of the Binet formula. One can ask whether our sequence $(h_{n})_{n\in\N}$ also satisfies some non-trivial identities. Of course, due to the fact that we do not have simple closed form of $h_{n}$ and behavior of our sequence is quite complicated, we cannot expect many such identities. However, we still are able to prove something interesting.

To start let us recall that the sequence $(t_{n})_{n\in\N}$ satisfies $t_{2n}=t_{n}$ and $t_{2n+1}=-t_{n}$. Thus, for $n\geq 5$ we have
\begin{align*}
\frac{h_{2n}-h_{2n-2}}{h_{2n-1}}&=t_{2n}=t_{n}=\frac{h_{n}-h_{n-2}}{h_{n-1}},\\
\frac{h_{2n+1}-h_{2n-1}}{h_{2n}}&=t_{2n+1}=-t_{n}=-\frac{h_{n}-h_{n-2}}{h_{n-1}}.
\end{align*}

As a consequence of the above equalities, we see that the sequence $(h_{n})_{n\in\N}$ can be defined without the use of the PTM sequence via the recurrence:
\begin{equation}\label{withoutPTM}
h_{2n}=\frac{h_{2n-1}}{h_{n-1}}(h_{n}-h_{n-2})+h_{2n-2},\quad h_{2n+1}=\frac{h_{2n}}{h_{n-1}}(h_{n-2}-h_{n})+h_{2n-1}.
\end{equation}
 It should also be stressed that exactly the same identities are satisfied by the Fibonacci sequence. Let us also note that if $F(p,q,r,s)=r^2-q^2+qs-pr$ then $F(h_{2n-2}, h_{2n-1}, h_{2n}, h_{2n+1})=0$ and a similar identity holds for the Fibonacci sequence.

\begin{lem}
For $n\in\N$ we have the following summation identities:
$$
\sum_{i=1}^{n}t_{i}h_{2i-1}=h_{2n},\quad \sum_{i=1}^{n}t_{i}h_{2i}=1-h_{2i+1}.
$$
\end{lem}
\begin{proof}
The above formulas follow from the general identities:
\begin{align*}
r_{2n}-r_{0}&=\sum_{i=1}^{n}(r_{2i}-r_{2i-2})=\sum_{i=1}^{n}a_{2i}r_{2i-1},\\
r_{2n+1}-r_{1}&=\sum_{i=1}^{n}(r_{2i+1}-r_{2i-1})=\sum_{i=1}^{n}a_{2i+1}r_{2i},
\end{align*}
which hold for the sequence $(r_{n})_{n\in\N}$ governed by the recurrence $r_{n}=a_{n}r_{n-1}+r_{n-2}$. In our case we have $a_{n}=t_{n}, r_{0}=h_{0}=0, r_{1}=h_{1}=1$, and hence the result.
\end{proof}

One among many interesting properties of the Fibonacci sequence $(f_{n})_{n\in\N}$ is the identity $f_{2n+1}=f_{n}^2+f_{n+1}^2$ which holds for all $n\in\N$. As a consequence we know that all prime factors of $f_{2n+1}$ are congruent to $1\pmod*{4}$ and these which are congruent to $3\pmod*{4}$ appear with even exponents. Although such a general result is not true for $h_{2n+1}$ we were able to get the following.

\begin{thm}
We have the following identities
$$
h_{2^{2k+1}-1}^2+h_{2^{2k+1}+1}^2=h_{2^{2k+1}-2}^2+h_{2^{2k+1}-1}^2=h_{2^{2k+2}-3}.
$$
\end{thm}

\begin{proof}
Consider any sequence $(r_n)_{n\in \N}$ which satisfies the recurrence $r_n = t_nr_{n-1} + r_{n-2}$; then (after extending domain of indexes to $\N \cup \{-2,-1\}$) any term of the sequence $(r_{n})_{n\in\N}$ can be expressed as a linear combination of $r_{-2}, r_{-1}$. A similar statement holds for any sequence $(r'_n)_{n\in \N}$ following the recurrence $r'_n = -t_nr'_{n-1} + r'_{n-2}$. Hence, we can define a sequence $(a_n)_{n\in \N}$ with values being 4-by-4 matrices of the form
$$a_n = \begin{bmatrix}a_{n,1}&a_{n,2}&0&0\\ a_{n,3}&a_{n,4}&0&0 \\ 0&0&a_{n,5}&a_{n,6}\\ 0&0&a_{n,7}&a_{n,8}\end{bmatrix}$$
such that
$$[r_{2^n-1}, r_{2^n-2}, r'_{2^n-1}, r'_{2^n-2}]^T = a_n \cdot [r_{-1}, r_{-2}, r'_{-1}, r'_{-2}]^T.$$
Those matrices are independent of $(r_n)_{n\in \N \cup \{-2,-1\}},(r'_n)_{n\in \N \cup \{-2,-1\}}$ as long as the recurrences hold. Since for $k \in \{0,\ldots,2^n-1\}$ we have $t_{2^n+k} = -t_k$, the linear combination describing $r_{2^{n+1}-1}$ in terms of $r_{2^n-2}, r_{2^n-1}$ must be the same as linear combination describing $r'_{2^n-1}$ in terms of $r'_{-2}, r'_{-1}$. By a similar argument we have the equality

$$\begin{bmatrix} r_{2^{n+1}-1}\\ r_{2^{n+1}-2}\\ r'_{2^{n+1}-1}\\ r'_{2^{n+1}-2}\end{bmatrix} = \begin{bmatrix}a_{n,5}&a_{n,6}&0&0\\ a_{n,7}&a_{n,8}&0&0 \\ 0&0&a_{n,1}&a_{n,2}\\ 0&0&a_{n,3}&a_{n,4}\end{bmatrix} \cdot \begin{bmatrix}r_{2^n-1}\\ r_{2^n-2}\\ r'_{2^n-1}\\ r'_{2^n-2}\end{bmatrix},$$
hence
$$a_{n+1} = J a_n J a_n, \text{ where } J =\begin{bmatrix}0&1&0&0\\ 1&0&0&0 \\ 0&0&0&1\\ 0&0&1&0\end{bmatrix}.$$

By a direct computation we get the equalities
\begin{align*}
    a_{n+1,1} &=a_{n,1}a_{n,5}+a_{n,3}a_{n,6},\\
    a_{n+1,2} &=a_{n,2}a_{n,5}+a_{n,4}a_{n,6},\\
    a_{n+1,3} &=a_{n,1}a_{n,7}+a_{n,3}a_{n,8},\\
    a_{n+1,4} &=a_{n,2}a_{n,7}+a_{n,4}a_{n,8},\\
    a_{n+1,5} &=a_{n,1}a_{n,5}+a_{n,2}a_{n,7},\\
    a_{n+1,6} &=a_{n,1}a_{n,6}+a_{n,2}a_{n,8},\\
    a_{n+1,7} &=a_{n,3}a_{n,5}+a_{n,4}a_{n,7},\\
    a_{n+1,8} &=a_{n,3}a_{n,6}+a_{n,4}a_{n,8},
\end{align*}
and
$$a_0 = \begin{bmatrix}1&1&0&0\\ 1&0&0&0 \\ 0&0&-1&1\\ 0&0&1&0\end{bmatrix}.$$
It follows by a simple induction that for $n>0$ we have
\begin{align*}
    a_{n,5} &= a_{n,1},\\
    a_{n,8} &= a_{n,4},\\
    a_{n,6} &= -a_{n,2},\\
    a_{n,3} &= (-1)^na_{n,2},\\
    a_{n,7} &= (-1)^{n+1}a_{n,2}.
\end{align*}
Using the relations above we get simplified recurrences for remaining elements of $a_{n+1}$:
\begin{align*}
    a_{n+1,1} &=a_{n,1}^2+(-1)^{n+1}a_{n,2}^2,\\
    a_{n+1,2} &=a_{n,2}(a_{n,1}-a_{n,4}),\\
    a_{n+1,4} &=a_{n,4}^2+(-1)^{n+1}a_{n,2}^2.
\end{align*}
So far we obtained some information about $(a_n)_{n\in N}$ for general $(r_n)_{n\in \N}, (r'_n)_{n\in \N}$.
For $r_{-2}=-1, r_{-1}=1$ we have $(r_n)_{n\in \N} = (h_n)_{n\in \N}$ so
\begin{align*}
h_{2^{2k+1}-1} &= a_{2k+1,1}-a_{2k+1,2},\\
h_{2^{2k+1}-2} &= a_{2k+1,3}-a_{2k+1,4}=-a_{2k+1,2}-a_{2k+1,4},\\
h_{2^{2k+2}-1} &= a_{2k+2,1}-a_{2k+2,2},\\
h_{2^{2k+2}-2} &= a_{2k+2,3}-a_{2k+2,4}=a_{2k+2,2}-a_{2k+2,4}.
\end{align*}
From properties of Thue-Morse sequence $t_{2^{2k+2}-1}=1$ hence
\begin{align*}
    h_{2^{2k+2}-3}&=h_{2^{2k+2}-1}-h_{2^{2k+2}-2}\\
    &=a_{2k+2,1}-a_{2k+2,2}-a_{2k+2,2}+a_{2k+2,4}\\
    &=-2a_{2k+1,2}(a_{2k+1,1}-a_{2k+1,4})+a_{2k+1,1}^2+a_{2k+1,4}^2+2a_{2k+1,2}^2\\
    &=(a_{2k+1,1} - a_{2k+1,2})^2 + (a_{2k+1,4} + a_{2k+1,2})^2\\
    &=h_{2^{2k+1}-1}^2 + h_{2^{2k+1}-2}^2.
\end{align*}
The other identity follows from Lemma \ref{lem:tools}(1).
\end{proof}

We turn our attention to the behavior of (kind of) a continued fraction expansion of $h_{n}/h_{n-1}$ for $n\geq 5$. It is well known that $f_{n}/f_{n-1}=[1;1,\ldots,1]$, where we have $n-1$ appearances of $1$'s in the expansion. Our information for the fraction $h_{n}/h_{n-1}$ is not so precise but our initial result is the following.

\begin{prop}
For $n\in\N_{\geq 5}$ we have the following identity:
$$
\frac{h_{n}}{h_{n-1}}=[t_{n};t_{n-1},\ldots, t_{5}].
$$
\end{prop}
\begin{proof}
We know that $h_{n}\neq 0$ for $n\geq 4$ thus the value of $h_{n}/h_{n-1}$ for $n\geq 5$ is well defined. We proceed by induction on $n$. The statement is clearly true for $n=5$. Let us suppose that
$$
\frac{h_{n}}{h_{n-1}}=[t_{n};t_{n-1},\ldots, t_{5}],
$$
for some $n\geq 5$. We have $h_{n+1}=t_{n+1}h_{n}+h_{n-1}$ and thus
\begin{align*}
\frac{h_{n+1}}{h_{n}}&=t_{n+1}+\frac{h_{n-1}}{h_{n}}=t_{n+1}+\frac{1}{\frac{h_{n}}{h_{n-1}}}\\
                     &=t_{n+1}+\frac{1}{[t_{n};t_{n-1},\ldots, t_{5}]}=[t_{n+1};t_{n},t_{n-1},\ldots, t_{5}]
\end{align*}
and hence the result.
\end{proof}
\begin{rem}
Essentially a similar statement is true for a general recurrence $r_{n}=s_{n}r_{n-1}+r_{n-2}$ provided that $r_{n}\neq 0$ for $n$ sufficiently large. More precisely, if $r_{n}\neq 0$ for $n\geq N$ then one can write
$$
\frac{r_{n}}{r_{n-1}}=[s_{n};s_{n-1},\ldots,s_{N}].
$$

\end{rem}

The above result is not the best possible because we deal with a non-regular continued fraction. Although we were unable to obtain precise value of digits in the regular continued fraction expansion of $|h_{n}|/|h_{n-1}|$ we were able to get the following.

\begin{thm}\label{confra}
Let $n\in\N_{\geq 5}$. The continued fraction expansion of $|h_{n}|/|h_{n-1}|$ is of the form $[a(n);\varepsilon_{1}(n),\ldots,\varepsilon_{k}(n)]$, where $a(n)\in\{0, 1, 2, 3, 4\}$, and $\varepsilon_{i}(n)\in\{1, 2, 3\}$ for each $i\in\{1, \ldots, k_{n}\}$.
\end{thm}
\begin{proof}
We will prove a slightly stronger version, namely to the above we add
$$a(n)\in
\begin{cases}
\{0,1,2\} & \text{ for } h_n<0, h_{n-1}>0, t_n=1, \\
\{0,1,2\} & \text{ for } h_n<0, h_{n-1}<0, h_{n-2}>0, \\
\{0,1,2\} & \text{ for } h_n<0, h_{n-1}<0, h_{n-2}<0, t_n=-1, \\
\{0\} & \text{ for } h_n>0, h_{n-2}<0, \\
\{1,2,3,4\} & \text{ for } h_n<0, h_{n-1}>0, t_n=-1, \\
\{1,2,3,4\} & \text{ for } h_n<0, h_{n-1}<0, h_{n-2}<0, t_n=1,\\
\{1,2\} & \text{ for } h_n>0, h_{n-2}>0.
\end{cases}$$
For $n<10$ our claim can be verified directly. Assume it is true for all $5\le N<n$; then for $n$ we have seven cases to consider (as stated above). We perform case by case analysis as follows:
\begin{itemize}
    \item If $h_n<0, h_{n-1}>0, t_n=1$, then $h_n = h_{n-1}+h_{n-2}=|h_{n-1}|-|h_{n-2}|$ and $\frac{|h_n|}{|h_{n-1}|}=\frac{|h_{n-2}|-|h_{n-1}|}{|h_{n-1}|}=\frac{|h_{n-2}|}{|h_{n-1}|}-1$. We know that $\frac{|h_{n-1}|}{|h_{n-2}|}>\frac{1}{4}$ as its expansion contains only digits $\in \{0,1,2,3\}$. Because $|h_{n-1}|<|h_{n-2}|$, the continued fraction of $\frac{|h_{n-2}|}{|h_{n-1}|}$ is the same as of $\frac{|h_{n-1}|}{|h_{n-2}|}$ without the leading zero. Hence, $a_n<3$ and no digits greater than $3$ appear in the expansion.

    \item If $h_n<0, h_{n-1}<0, h_{n-2}>0$, then as above $\frac{|h_n|}{|h_{n-1}|}=\frac{|h_{n-2}|}{|h_{n-1}|}-1$.

    \item The case when $h_n<0, h_{n-1}<0, h_{n-2}<0, t_n=-1$ leads to the same conclusion using the same arguments as above.

    \item Tf $h_n>0, h_{n-2}<0$ then using Lemma \ref{lem:tools} we get $2 \nmid n, t_n=-1, t_{n-2}=1, h_{n-1} <0, h_{n-3}<0, |h_n|=|h_{n-3}|$. Then $h_{n-1}=h_{n-2}+h_{n-3}, |h_{n-1}|=|h_{n-2}|+|h_{n-3}|$ and
    $$\frac{|h_n|}{|h_{n-1}|}=\frac{|h_{n-3}|}{|h_{n-2}| + |h_{n-3}|}=\frac{1}{1+\frac{|h_{n-2}|}{|h_{n-3}|}}.$$
    Hence $a(n)=0$ and as $a(n-2)\le 2$ there are no digits greater than $3$ in the continued fraction of $\frac{|h_n|}{|h_{n-1}|}$.

    \item When $h_n<0, h_{n-1}>0, t_n=-1 $ we have $h_n = -h_{n-1}+h_{n-2}= -|h_{n-1}|-|h_{n-2}|$, so
     $$\frac{|h_n|}{|h_{n-1}|}=\frac{|h_{n-1}| + |h_{n-2}|}{|h_{n-1}|}=1 + \frac{1}{\frac{|h_{n-1}|}{|h_{n-2}|}}.$$
     As $h_{n-1}>0$ we have $a(n-1) < 4$, so no digits greater than 3 can appear in this continued fraction.

     \item In the case $h_n<0, h_{n-1}<0, h_{n-2}<0, t_n=1$, we again have
     $\frac{|h_n|}{|h_{n-1}|}=1 + \frac{1}{\frac{|h_{n-1}|}{|h_{n-2}|}}$. This time $a(n-1)$ cannot be equal to $4$ because $h_{n-2}>0$ and due to Lemma \ref{lem:tools} $t_{n-1}, t_n$ cannot be both equal to $1$ when $h_n, h_{n-1}$ are negative.

     \item The last possibility is $h_n>0, h_{n-2}>0$; again  $\frac{|h_n|}{|h_{n-1}|}=1 + \frac{1}{\frac{|h_{n-1}|}{|h_{n-2}|}}$. From Lemma \ref{lem:tools} it follows that $2\nmid n, t_n=-1, t_{n-1}=1, h_{n-1}<0$, so $a(n-1)<3$ and again no digits greater than 3 can appear in the expansion of $\frac{|h_n|}{|h_{n-1}|}$. We still need to show that $a(n)<3$.
     From Lemma~\ref{lem:tools} we can get $t_{n-4}=1, h_{n-4}<0$ as $h_{n-2}>0, h_n>0$. Moreover, we know that the PTM sequence $(t_n)_{n\in \N}$ has no three consecutive 1's. Repeating calculations for proper cases we get
     \begin{align*}
     \frac{|h_{n-1}|}{|h_{n-2}|}&=\frac{|h_{n-3}|}{|h_{n-2}|}-1,\\
     \frac{|h_{n-2}|}{|h_{n-3}|}&=\frac{1}{1+\frac{|h_{n-4}|}{|h_{n-5}|}},\\
     \frac{|h_n|}{|h_{n-1}|} &= 1+\frac{|h_{n-5}|}{|h_{n-4}|}.
     \end{align*}
     If $a(n)>2$ then necessarily $a(n-4)=0$. This equality implies $h_{n-6}>0$ (we fall into the second case) and
     $$\frac{|h_{n-4}|}{|h_{n-5}|}=\frac{|h_{n-6}|}{|h_{n-5}|}-1.$$
     Again one can infer from properties of the PTM sequence that $t_{n-5} = -1$ so $a(n-5) \in \{1,2,3,4\}$, which would lead to the negative value of $\frac{|h_{n-4}|}{|h_{n-5}|}$, a contradiction.

\end{itemize}
\end{proof}

\section{Questions and conjectures}\label{sec5}

In this section we collect some questions and conjectures which appeared during our investigations.

Let us return for a moment to identities (\ref{withoutPTM}). As we observed, the same identities are satisfied by the Fibonacci sequence. This is a bit surprising and one can formulate the following

\begin{ques}
Let $k\in\N_{\geq 2}$ be given and put $X_{i}=\{x_{i,1},\ldots, x_{i,m_{i}}\}$ for $i=1, \ldots, k$. Characterize those polynomials $F\in\Z[X_{1},\ldots, X_{k}]$ satisfying the following:
\begin{align*}
\forall n\in\N:\;F&(f_{n},\ldots, f_{n+i_{1}},\ldots,f_{kn},\ldots, f_{kn+i_{k}})=0\\
                  &\;\Longrightarrow\;F(h_{n},\ldots, h_{n+i_{1}},\ldots,h_{kn},\ldots, h_{kn+i_{k}})=0.
\end{align*}
\end{ques}

Note that $f_{2n}\equiv 0\pmod*{f_{n}}$ for $n\in\N_{+}$. Although a similar property is not true in the case of the sequence $(h_{n})_{n\in\N}$ our numerical calculations suggest the following.

\begin{conj}
The following congruences are true:
\begin{align*}
&h_{2(2^{2n}-3)}\equiv 0\pmod*{h_{2^{2n}-3}},\;n\in\N_{+},\\
&h_{2(2^{2n-1}-3)}\equiv -2\pmod*{h_{2^{2n+1}-3}},\;n\in\N_{\geq 2}.
\end{align*}
\end{conj}
We also formulate the following

\begin{ques}
Let $m\in\N_{\geq 2}$ and put
$$
V_{m}:=\{h_{n}\pmod*{m}:\;n\in\N\}, \quad I_{m}:=\{0,\ldots, m-1\}.
$$
For which values of $m$ do we have $V_{m}=I_{m}$?
\end{ques}
This is an interesting question and our numerical computation confirms that there are values of $m$ such that $V_{m}\neq I_{m}$. We checked that for $m\leq 100$ there are $m=33, 54, 66, 83, 99$. We believe that the set of such values is infinite and it would be interesting to have a characterization of them.

Through the paper we observed some similarities and differences between the Fibonacci sequence and the sequence $(h_{n})_{n\in\N}$. One can ask the following

\begin{ques}
Is the set
$$
\cal{I}=\{f_{n}:\;n\in\N\}\cap \{h_{n}:\;n\in\N\}
$$
finite?
\end{ques}
We expect that the answer is YES. In fact, we believe that $\cal{I}=\{0, 1, 3, 5\}$.

In Theorem \ref{confra} we proved that if
$$
\frac{|h_{n}|}{|h_{n-1}|}=[a(n);\varepsilon_{1}(n),\ldots,\varepsilon_{k}(n)]$$
then $a(n)\in\{0,1,2,3,4\}$. Note that $a(n)=\lfloor|h_{n}|/|h_{n-1}|\rfloor$. In fact we believe that an even stronger property is true. More precisely, we believe that the following is true.

\begin{conj}
The sequence $(a(n))_{n\in\N}$ is 2-automatic.
\end{conj}



Our next conjecture is related to the behavior of the sequence of the fractions $|h_{n}|/|h_{n-1}|$ along subsequences of the form $2^{k}n+i, i=0,\ldots, 2^{k}-1$. The pictures presented below strongly suggest the following:

\begin{figure}[h]
       \centering
         \includegraphics[width=5in]{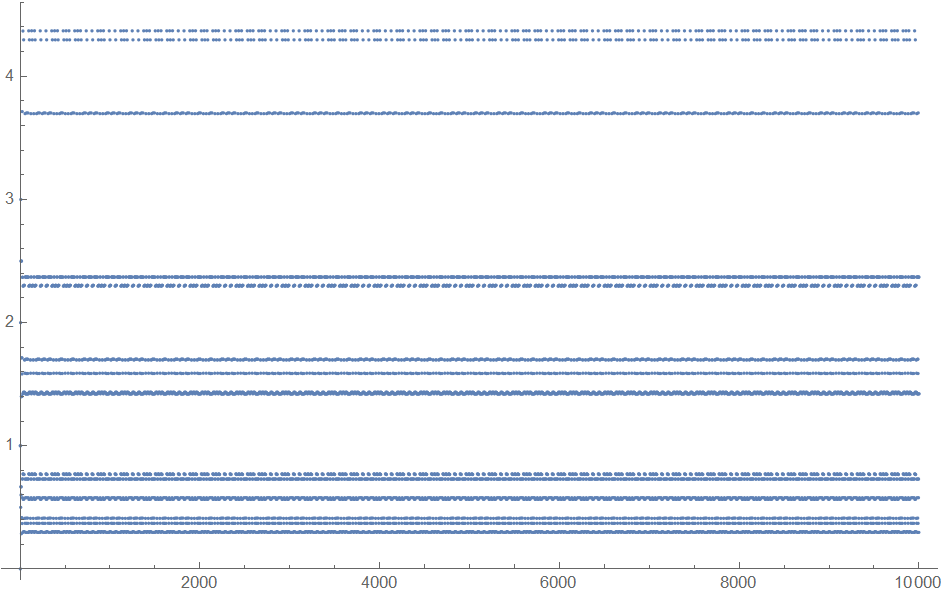}
        \caption{The plot of the sequence $|h_{n}|/|h_{n-1}|, n\leq 10^4$}
       \label{fraction_1}
    \end{figure}

\begin{figure}[h]
       \centering
         \includegraphics[width=5in]{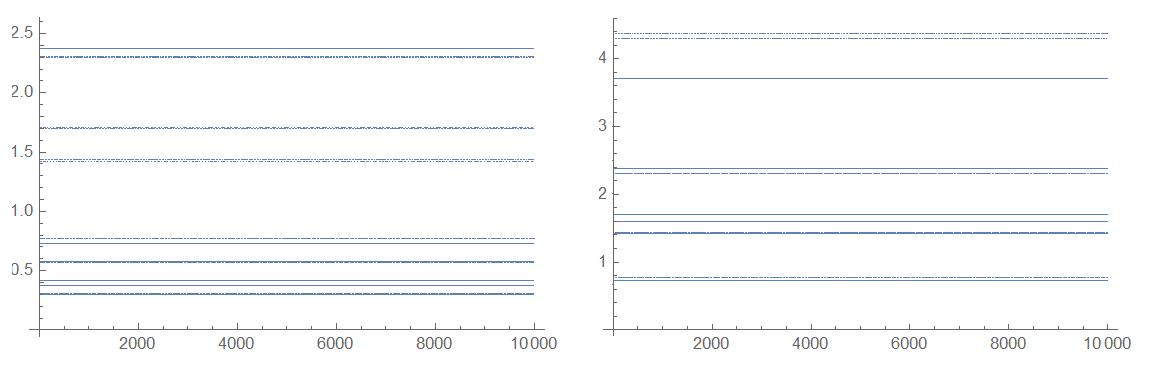}
        \caption{The plot of the sequence $|h_{2n+i+1}|/|h_{2n+i}|, n\leq 10^4$ for $i=0$ (left) and $i=1$ (right)}
       \label{fraction_2}
    \end{figure}

\begin{figure}[h]
       \centering
         \includegraphics[width=5in]{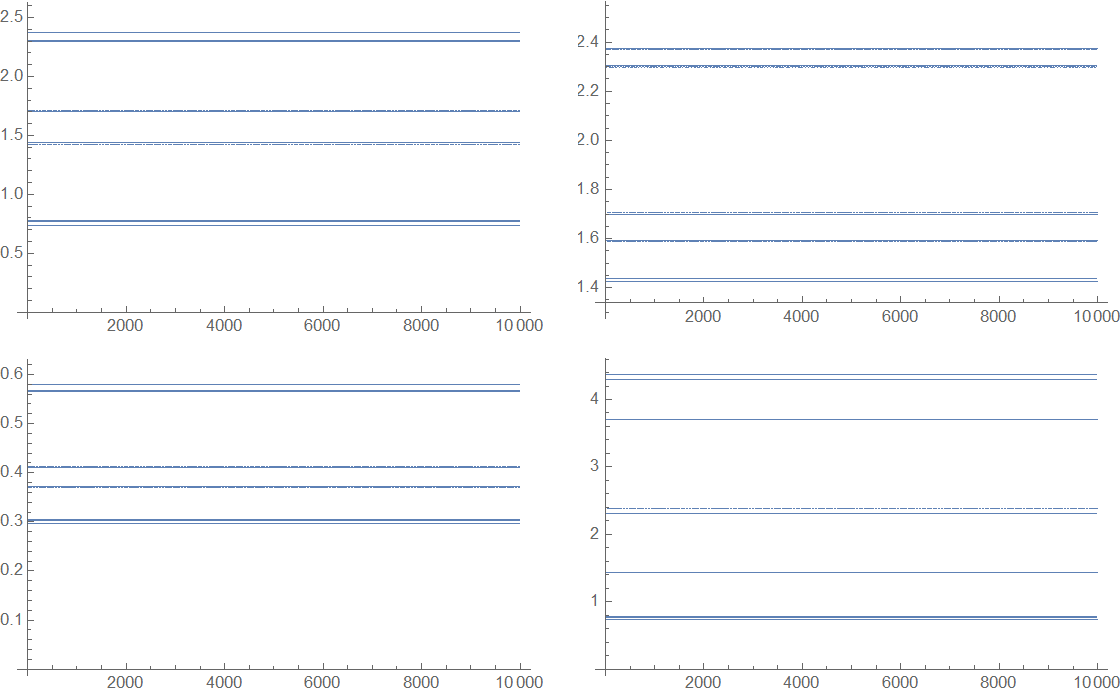}
        \caption{The plot of the sequence $|h_{4n+i+1}|/|h_{4n+i}|, n\leq 10^4$ for $i=0$ (upper left), $i=1$ (upper right), $i=2$ (lower left), $i=3$ (lower right)}
       \label{fraction_4}
    \end{figure}

\begin{conj}
For $n\in \N_{\geq 4}$ we have $|h_{n}|\geq \frac{2}{7}|h_{n-1}|$. It seems that more is true. More precisely, for any $k\in\N$ and $i\in\{0,\ldots, 2^{k}-1\}$ there is a constant $C_{k,i}\in\Q$ such that
$$
|h_{2^{k}n+i+1}|\geq C_{k,i}|h_{2^{k}n+i}|
$$
for each $n\in\N$. In particular, we have $C_{2k,0}=C_{2k+1,0}$ for each $k$. In particular, we have the equalities
\begin{align*}
C_{0,0}&=C_{1,0}=\frac{2}{7},\\
C_{2,0}&=C_{3,0}=\frac{2}{3},\\
C_{4,0}&=C_{5,0}=\frac{64}{83},\\
C_{6,0}&=C_{7,0}=\frac{52071130}{67519091},
\end{align*}
and for each $k\in\N_{+}$ we have
$$
\liminf_{n\rightarrow +\infty}\frac{|h_{2^{k}n+i+1}|}{|h_{2^{k}n+i}|}=C_{k,i}.
$$
\end{conj}

\bigskip

\begin{ques}
What is the true order of magnitude of $|h_{n}|$ as $n\rightarrow +\infty$? More precisely, does the limit $\lim_{n\rightarrow +\infty}|h_{n}|^{\frac{1}{n}}$ exist?
\end{ques}

We believe that the limit of the sequence $(|h_{n}|^{\frac{1}{n}})_{n\in\N_{+}}$ exists and is equal to $g\approx 1.152$. The conjectural value of $g$ follows from our computations of the first $10^{5}$ terms.

\bigskip

We studied Fibonacci-like sequence twisted by the PTM sequence which is 2-automatic. In particular, we proved that the sequence of signs of $(h_{n})_{n\in\N}$ is 2-automatic (and it is not periodic). This suggests the following general

\begin{ques}
Let $(a_{n})_{n\in\N}$ be an automatic sequence and consider the sequence $(r_{n})_{n\in\N}$ satisfying $r_{0}=0, r_{1}=1$ and $r_{n}=a_{n}r_{n-1}+r_{n-2}, n\geq 2$. Is the sequence $(\op{sign}(r_{n}))_{n\in\N}$ automatic?
\end{ques}

Our expectation is that the answer to the question above is positive. We performed some experiments with some automatic sequences and in each case the sequence of signs was indeed automatic. For example, let us take $m\in\N_{\geq 2}$ and consider the sequence
$$a_{n}(m)=\begin{cases}
+1, & n\;\mbox{is a power of}\;m,  \\
-1, & \mbox{otherwise},
\end{cases}$$
i.e., the sequence $(a_{n}(m))_{n\in\N}$ can be seen as a characteristic sequence (where the usual value of 0 is replaced by $-1$) of the set of powers of $m$. If $r_{0}(m)=0, r_{1}(m)=1$ and for $n\geq 2$ we have $r_{n}(m)=a_{m}(n)r_{n-1}(m)+r_{n-2}(m)$, then one can prove that for $n\geq 2m+1$ we have
$$
\op{sign}(r_{n}(m))=\begin{cases}               (-1)^{n+1}a_{n}(2), & m=2, \\
(-1)^{n}a_{n}(m)  , & m>2.
\end{cases}
$$
In particular, for each $m\in\N_{\geq 2}$ the sequence $(\op{sign}(r_{n}(m)))_{n\in\N}$ is $m$-automatic.

\section{Acknowledgements} 

Work of both authors was supported by a grant of the National Science Centre (NCN) No.\ UMO-2019/34/E/ST1/00094.

Thanks to Agnieszka Dutka for proofreading and {\LaTeX } support.

 \bibliographystyle{elsarticle-num} 
 \bibliography{cas-refs}

\end{document}